\documentclass[reqno, 12pt]{amsart}
\usepackage{enumerate}
\usepackage{amssymb,url,color}
\usepackage{hyperref}
\usepackage{amsmath,amsthm}
\usepackage{amsfonts}
\usepackage{mathrsfs}
\usepackage{stmaryrd}
\usepackage{graphicx}
\usepackage{bbm}

\allowdisplaybreaks[4]

\newtheorem{thm}{Theorem}[section]
\newtheorem{cor}{Corollary}[section]
\newtheorem{prop}{Proposition}[section]
\newtheorem{lem}{Lemma}[section]
\newtheorem{rem}{Remark}[section]
\newtheorem{exa}{Example}[section]
\theoremstyle{Problem}

\theoremstyle{definition}
\newtheorem{defn}{Definition}[section]
\numberwithin{equation}{section}
\newcommand{\pp}{\mathbb{P}}
\newcommand{\ee}{\mathbb{E}}

\newcommand{\rr}{\mathbb{R}}
\newcommand{\Var}{\mathrm{Var}}
\def\beq{\begin{equation}}
\def\deq{\end{equation}}
\allowdisplaybreaks 

\bfseries\rmfamily 

\begin{document}

\title [Berry-Esseen bounds for entropy and diversity indices]
{Berry-Esseen bounds for estimators of entropy and diversity indices on countable alphabets}
\thanks{This work is supported by Natural Science Foundation of Henan (No. 262300421307).}

\author[Z. H. Yu]{Zhenhong Yu}
\address[Z. H. Yu]{School of Mathematics and Statistics, Henan Normal University, Henan Province, 453007, China.} \email{\href{mailto: Z. H. Yu
<zhenhongyu2022@126.com>}{zhenhongyu2022@126.com}}

\author[Y. Miao]{Yu Miao}
\address[Y. Miao]{School of Mathematics and Statistics, Henan Normal University, Henan Province, 453007, China.} \email{\href{mailto: Y. Miao
<yumiao728@gmail.com>}{yumiao728@gmail.com}; \href{mailto: Y. Miao <yumiao728@126.com>}{yumiao728@126.com}}

\begin{abstract}
In the present paper, we derive Berry-Esseen bounds for the estimation of diversity indices on countable alphabets.  A general non-asymptotic convergence rate is established for the plug-in estimator of a wide class of indices, including Simpson's index and Re\'{n}yi's entropy. For the practically crucial case of Shannon entropy, we provide explicit Berry-Esseen bounds for the standard plug-in estimator, as well as for two widely used bias-corrected variants, the Miller-Madow and the jackknife estimators.
\end{abstract}

\maketitle

\section{Introduction}
Diversity indices provide the fundamental framework for quantifying the uncertainty or information content of a probability distribution over an alphabet. In this setting, the absence of natural ordering renders classical moments like variance and standard deviation undefined.
Beyond information theory, these indices are widely applied in ecology, genetics, and linguistics, where they quantify the complexity of populations such as species assemblages, cancer cell types, and author vocabularies. The earliest contributions to this field include Shannon's entropy \cite{Shannon} and Simpson's index \cite{Simpson}, both of which remain widely used. A unifying framework for these indices was established by R\'{e}nyi \cite{Renyi}, who introduced a parametric family of entropies that includes both Shannon's entropy and Simpson's index as special cases. This framework was further advanced by Hill \cite{Hill}, who showed that a broad class of diversity indices are equivalent under monotonic transformations. Zhang and Grabchak \cite{Z-G} showed that a large class of diversity indices in the literature can be represented as linear combinations of an entropic basis and proposed corresponding nonparametric estimators.

The most prevalent approach for evaluating a diversity index is the plug-in estimator. While extensive research has addressed its asymptotic properties for finite alphabets \cite{Z-G}, the case of countably infinite alphabets presents additional challenges. In previous papers, Grabchak and Zhang \cite{G-Z} studied the asymptotic distribution of the plug-in estimator for a large class of diversity indices. Additionally, Paninski \cite{P} and Zhang and Zhang \cite{ZhZh12} proved the asymptotic normality of the plug-in estimator of Shannon's entropy, and Zhang and Shi \cite{Z-S} extended this analysis to generalized Shannon entropy. To further reduce the sample bias, a number of bias-corrected modifications have been proposed in Pinchas et al. \cite{A}.
Among these, two of the most popular are the Miller-Madow estimator \cite{M}, which applies an additive correction based on the number of observed symbols, and the jackknife estimator \cite{Z-J}, a resampling-based technique that systematically reduces bias by recomputing the estimate on subsamples of the data. Chen et al. \cite{C-M} provided sufficient conditions for the asymptotic normality of these two bias-corrected estimators. However, these studies primarily establish asymptotic distributions without providing explicit rates of convergence.

In the present paper, we establish Berry-Esseen bounds for a large class of diversity indices and their bias-corrected estimators, thereby offering explicit non-asymptotic convergence rates. Section 2 states the main results. Our results hold for a wide family of diversity indices including Simpson's index and R\'{e}nyi's entropy. While Shannon's entropy \cite{Shannon} does not directly fall into our framework, we provide a separate proof for this specific case. The result is then extended to the Miller-Madow and jackknife estimators of Shannon's entropy. In Section 3, we discuss some examples to show that our conditions are readily satisfied. The proofs of all results are detailed in Section 4.
\section{Diversity Indices}\label{1}
We first present a general theorem  establishing the non-asymptotic convergence rate of the plug-in estimator. Specifically, our analysis covers a triangular array setting, where the underlying distribution may depend on the sample size $n$, as well as the classical independent and identically distributed setting with a fixed distribution. To state our results, we need the following definition.
\begin{defn}\label{def1}
Fix $\beta\in(0,1]$. A function $g:[0,1] \to \rr$ is called $\beta$-H\"older continuous if there exists a constant $M > 0$ such that, for any $x,y \in [0,1]$, we have
$$
|g(x)-g(y)|\leq M|x-y|^\beta.
$$
\end{defn}
It is easy to see that every $\beta$-H\"older continuous function is continuous and bounded on a closed interval.

Let $\mathcal{A}=\{a_i, i\ge 1\}$ be a countably infinite alphabet with associated probability measures $\mathbf{P}_n=\{p_{n,i}, i\ge 1\}$ for each $n$. The letters of $\mathcal{A}$ correspond to species in an ecosystem, words in the English language, types of cancer cells in a tumour, or any other categorical quantity whose diversity is of interest to us. We allow some (even countably many) $p_{n,i}$s to be zero. Thus finite alphabets are a special case of this model.

For each $n$, diversity index is a function $\theta$ that maps $\mathbf{P}_n$ into $\rr$. A common assumption is that
\beq\label{theta}
\theta_n=\theta(\mathbf{P}_n)=\sum_{i=1}^\infty g\left(p_{n,i}\right),
\deq
where $g: [0,1]\to \rr$. Such indices were called dichotomous indices in Patil and Taillie \cite{P-T}. To ensure that the index is well defined, we assume that
\beq\label{theta1}
\sum_{i=1}^\infty \left|g\left(p_{n,i}\right)\right|<\infty,\ \ \ \text{for each}\ n.
\deq
For each $n\ge 1$, let $\{X_{k,n}, 1\le k\le n\}$ be an array of independent and identically distributed random variables taking values in some countably infinite alphabet $\mathcal{A}$ with common distribution $\mathbf{P}_n$, i.e.,
$$
p_{n,i}=\pp(X_{1,n}=a_i),  \ \ \  \ i\ge 1, \ n\ge 1.
$$
For each $i$, let
$$
\hat p_{n,i}:=\frac{1}{n}\sum_{k=1}^nI_{\{X_{k,n}=a_i\}}
$$
be the sample proportion. The plug-in estimator of $\theta_n$ is given by
\beq\label{theta2}
\hat\theta_n=\sum_{i=1}^\infty g(\hat p_{n,i}).
\deq
\begin{thm}\label{thmtheta}
Assume that the function $g:[0,1] \to \rr$ is differentiable and its derivative $g'$ is $\beta$-H\"{o}lder
continuous for some $\beta \in (0,1]$. Let
\beq\label{thmtheta-1}
\sigma^2_n=\sum_{i=1}^{\infty}p_{n,i}\left(g'\left( p_{n,i}\right)\right)^2-\left(\sum_{i=1}^{\infty}p_{n,i}g'\left( p_{n,i}\right)\right)^2
\deq
and
\beq\label{thmtheta-2}
\gamma=
\begin{cases}
\beta/2 \ \ &\text{if}\ \ \beta\in (0,0.5]\\
\beta-0.5 \ \ &\text{if}\ \ \beta\in (0.5,1]
\end{cases}.
\deq
If
\begin{enumerate}[$(1)$]
\item $\beta\in (0.5,1]$ or
\item $\beta\in (0,0.5]$ and
\beq\label{thm1-c}
\sup_{n}\sum_{i=1}^\infty p_{n,i}^{(\beta+1)/2}<\infty,
\deq
\end{enumerate}
then we have
\beq\label{thmtheta-3}
\sup_x\left|\pp\left(\frac{\sqrt{n}}{\sigma_n}(\hat \theta_n-\theta_n)\le x\right)-\Phi(x)\right|\le\frac{C}{n^{\gamma/2} \sigma_n^{1/2}},
\deq
where $\Phi$ denotes the standard normal distribution function.
\end{thm}
\begin{rem}
The bound directly implies the central limit theorem established in \cite{G-Z} with the same condition $$\lim_{n\to\infty} n^\gamma\sigma_n=\infty.$$
\end{rem}
Theorem \ref{thmtheta} allows the underlying distribution $\mathbf{P}_n$ to change with $n$. An important special case is the classical setting with fixed distribution.

Let $\{X_n, n\ge 1\}$ be a sequence of independent and identically distributed random variables taking values in alphabet $\mathcal{A}=\{a_i, i\ge 1\}$ with distribution
$\mathbf{P}=\{p_{i}, i\ge 1\}$ i.e.,
$$
p_{i}=\pp(X_{1}=a_i),  \ \ \  \ i\ge 1.
$$
For each $i$, let
\beq\label{hat}
\hat p_{i}:=\frac{1}{n}\sum_{k=1}^nI_{\{X_{k}=a_i\}}
\deq
be the sample proportion. The plug-in estimator of $\theta$ is given by
\beq\label{theta3}
\hat\theta_n=\sum_{i=1}^\infty g(\hat p_{i}).
\deq
\begin{cor}\label{cortheta}
Assume that the function $g$ is differentiable on $[0,1]$ and its derivative $g'$ is $\beta$-H\"{o}lder
continuous for some $\beta \in (0,1]$. Let
\beq\label{cortheta-1}
\sigma^2=\sum_{i=1}^{\infty}p_{i}\left(g'\left( p_{i}\right)\right)^2-\left(\sum_{i=1}^{\infty}p_{i}g'\left( p_{i}\right)\right)^2
\deq
and assume that
\beq\label{cortheta-2}
\gamma=
\begin{cases}
\beta/2 \ \ &\text{if}\ \ \beta\in (0,0.5]\\
\beta-0.5 \ \ &\text{if}\ \ \beta\in (0.5,1]
\end{cases}.
\deq
If
\begin{enumerate}[$(1)$]
\item $\beta\in (0.5,1]$ or
\item $\beta\in (0,0.5]$ and
$$
\sum_{i=1}^\infty p_{i}^{(\beta+1)/2}<\infty,
$$
\end{enumerate}
then we have
\beq\label{cortheta-3}
\sup_x\left|\pp\left(\frac{\sqrt{n}}{\sigma}(\hat\theta_n-\theta)\le x\right)-\Phi(x)\right|\le
\frac{C}{n^{\gamma/2}},
\deq
where $\Phi$ denotes the standard normal distribution function.
\end{cor}
\begin{rem}
Corollary \ref{cortheta} follows directly from Theorem \ref{thmtheta} by taking $\mathbf{P}_n \equiv \mathbf{P}$. It presents the bound in its classical form, where the distribution does not depend on $n$.
\end{rem}
\begin{rem}
The uniform bound in Corollary \ref{cortheta} implies the convergence in distribution of $\sqrt{n}(\hat \theta_n-\theta)/\sigma$ to the standard normal as $n\to\infty$. Thus, under the same conditions, Corollary \ref{cortheta} directly yields the central limit theorem established in \cite{G-Z}.
\end{rem}
Consider the index
\beq\label{h}
h_{\mu,\nu}=\sum_{i=1}^\infty p_i^\mu(1-p_i)^\nu
\deq
for $\mu > 0$ and $\nu \ge 0$. When $\mu = 2$ and $\nu = 0$, this reduces to Simpson's index introduced in Simpson \cite{Simpson}. When $\mu$ and $\nu$ are integers, this corresponds to the generalized Simpson's indices introduced in Zhang and Zhou \cite{ZhZh10} and further studied in Grabchak et al. \cite{G}.
When $\mu > 0$ and $\nu = 0$, this corresponds to the R\'enyi equivalent entropy introduced in Zhang and Grabchak \cite{Z-G}.

Note that for $h_{\mu,\nu}$, where $\mu > 0$ and $\nu \ge 0$, we have
$g(x)=x^\mu(1-x)^\nu$ and
$$
g'(x)=\mu x^{\mu-1}(1-x)^\nu-\nu x^\mu(1-x)^{\nu-1}.
$$
Furthermore, we recall the following properties.
\begin{prop}\label{prop1} \cite[Proposition 3.1]{G-Z} When $\mu\ge 1$ and $\nu\in\{0\}\cup [1,\infty)$, $g'$ is $\beta$-H\"older continuous with
$$
\beta=\begin{cases} \min\{\mu-1, \nu-1, 1\}\ \ & \text{if}\ \ \mu, \nu >1\\
\min\{\mu-1,  1\}\ \ & \text{if}\ \ \mu>1, \nu\in\{0,1\}\\
\min\{\nu-1, 1\}\ \ & \text{if}\ \ \mu=1, \nu >1\\
1\ \ & \text{if}\ \ \mu=1,  \nu\in\{0,1\}\\
\end{cases}
$$
\end{prop}
In order to illustrate the convergence rate more clearly, a concrete example for Theorem \ref{thmtheta} is given below.
\begin{exa}
Consider a sequence of distributions of the form
$$p_{n,1}=\frac{1}{2}+\frac{1}{2n^\lambda}, \ \ \ p_{n,2}=\frac{1}{2}-\frac{1}{2n^\lambda},$$
where $\lambda \in (0,1/2)$ is a real number and $p_{n,i}=0$ for all $i=3,4,\cdots$. Clearly, this approaches a uniform distribution as $n\to \infty$.  Suppose we consider Simpson's diversity index, which
corresponds to $g(x)=x^2$. In this case, $g'(x)=2x$ is $1$-H\"{o}lder continuous. Thus
$$
\aligned
\sigma_n^2=&\frac{1}{2}\left(1+\frac{1}{n^\lambda}\right)^3+\frac{1}{2}\left(1-\frac{1}{n^\lambda}\right)^3
-\frac{1}{4}\left(\left(1+\frac{1}{n^\lambda}\right)^2+\left(1-\frac{1}{n^\lambda}\right)^2\right)^2
\\
=&\frac{1}{n^{2\lambda}}-\frac{1}{n^{4\lambda}}\sim \frac{1}{n^{2\lambda}}.
\endaligned
$$
For the case we have
$$\sup_x\left|\pp\left(\frac{\sqrt{n}}{\sigma_n}(\hat \theta_n-\theta_n)\le x\right)-\Phi(x)\right|\le
\frac{C}{n^{1/2-\lambda/2}}.$$
\end{exa}
\section{Shannon's entropy and its bias-corrected estimators}
We now turn to Shannon’s entropy, one of the most well-known diversity indices. It is an exception to the general framework established in Section \ref{1} since its function $g(x) = -x \ln x$ has an unbounded derivative near zero. In this case, the assumptions of Corollary \ref{cortheta} do not hold, thus requiring a distinct analytical approach to establish its Berry-Esseen bound.

Let the sequence $\{X_n, n\ge 1\}$ and the distribution
$\mathbf{P}$ be defined as before. Shannon's entropy \cite{Shannon} was introduced as
$$
H=-\sum_{i=1}^\infty p_i\ln p_i,
$$
and its plug-in estimator is given by
\beq\label{plug}
\hat H_n=-\sum_{i=1}^\infty \hat p_i\ln \hat p_i
\deq
where
$$
\hat p_{i}:=\frac{1}{n}\sum_{k=1}^nI_{\{X_{k}=a_i\}}.
$$

Asymptotic normality for the plug-in estimator of Shannon's entropy on a countably infinite alphabet was proved by Zhang and Zhang \cite{ZhZh12}, relying on the result of Paninski \cite{P}. Building upon this foundation, we extend the analysis to obtain a non-asymptotic convergence rate.
\begin{thm}\label{thmH}
Fix $0< \delta< \frac{1}{2}$. Let $\{X_n, n\ge 1\}$ be a sequence of independent and identically distributed random variables with a common nonuniform distribution $\{p_i, i\ge 1\}$ satisfying $\ee|\ln p(X_1)|^{2+\delta} < \infty$. Assume that there exists an integer valued function $K(n)$ such that,
\beq\label{thmH-0}
K(n)\le n^{\frac{1}{2}-\delta},\ \ n^{\frac{1}{2}+\delta}\ln n\sum_{i=K(n)}^{\infty}p_i \leq C
\ \ \text{and}\ \ -n^{\frac{1}{2}+\delta}\sum_{i=K(n)}^{\infty}p_i\ln p_i\leq C,
\deq
then we have
\beq\label{thmH-1}
\sup\limits_x\left|\pp\left(\frac{\sqrt{n}(\hat H_n-H)}{\sigma}<x\right)-\Phi(x)\right|
 \leq Cn^{-\frac{\delta}{2}},
\deq
where $\Phi$ denotes the standard normal distribution function and $\sigma^2=\Var(\ln p(X_1))$.
\end{thm}
\begin{rem}
If $\delta=0$, the conditions in Theorem \ref{thmH} reduce to $K(n)\le n^{1/2}$ and
$$
n^{\frac{1}{2}+\delta}\ln n\sum_{i=K(n)}^{\infty}p_i \leq C,
\ \ \ -n^{\frac{1}{2}+\delta}\sum_{i=K(n)}^{\infty}p_i\ln p_i\leq C.
$$

In this case, the theorem does not provide an explicit convergence rate, but the asymptotic normality still holds. Indeed, under the assumption $\ee|\ln p(X_1)|^2 < \infty$  and the given conditions, the Lindeberg condition can be verified for the sequence, which implies the central limit theorem \cite{ZhZh12}. The absence of the rate is due to the technical limitations of the proof rather than the failure of convergence.
\end{rem}
Next we consider two widely used bias-corrected estimators: the Miller-Madow estimator and the jackknife estimator. Our goal is to establish Berry-Esseen bounds for these estimators, quantifying their convergence rates under distributional conditions similar to those in Theorem \ref{thmH}.

The Miller-Madow estimator \cite{M} applies the classic strategy of first-order bias correction through a simple adjustment term based on the number of observed symbols. Let $Y_i=\sum_{k=1}^nI_{\{X_{k}=i\}}$ be the observed letter counts in the sample, and let $\hat m=\sum_{i=1}^\infty I_{\{Y_i>0\}}$ be the number of letters observed in the sample. Then the Miller-Madow entropy estimator can be written as
$$\hat H_{MM}=\hat H_n+\frac{\hat m-1}{2n},$$
where $\hat H_n$ denotes the plug-in estimator of Shannon's entropy as defined in \ref{plug}.
\begin{thm}\label{thmM}
Let $\{X_n, n\ge 1\}$ be a sequence of independent and identically distributed random variables with common nonuniform distribution $\{p_i, i\ge 1\}$ satisfying
$\ee|\ln p(X_1)|^{2+\delta} < \infty$. Assume that there exist an integer valued function $K(n)$  and $0< \delta< \frac{1}{2}$ such that
\beq\label{thmM-1}
K(n)\le n^{\frac{1}{2}-\delta},\ \ n^{\frac{1}{2}+\delta}\ln n\sum_{i=K(n)}^{\infty}p_i \leq C,
\ \ -n^{\frac{1}{2}+\delta}\sum_{i=K(n)}^{\infty}p_i\ln p_i\leq C,
\deq
then we have
\beq\label{thmM-3}
\sup_x\left|\pp\left(\frac{\sqrt{n}(\hat{H}_{MM}-H)}{\sigma}<x\right)-\Phi(x)\right|
 \leq Cn^{-\frac{\delta}{2}},
\deq
where $\Phi$ denotes the standard normal distribution function and $\sigma^2=\Var(\ln p(X_1))$.
\end{thm}
The jackknife entropy estimator proposed by \cite{Z-J} is another commonly used estimator designed to reduce the bias of the plug-in estimator. For each $j\in\{1,2,\dots,n\}$, let $\hat H_n^{(j)}$ be the plug-in estimator computed from the subsample obtained by deleting the $j$th observation. That is,
$$
\hat H_n^{(j)}=-\sum_{i=1}^\infty \hat p_i^{(j)}\ln \hat p_i^{(j)},
$$
where
$$
\hat p_i^{(j)}=\frac{1}{n-1}\left(\sum_{k=1}^n I_{\{X_k=i\}}-I_{\{X_j=i\}}\right).
$$
The jackknife entropy estimator can then be written as
$$
\hat H_{JK}=n\hat H_n-(n-1)\frac{\sum_{j=1}^n \hat H^{(j)}}{n}.
$$
\begin{thm}\label{thmJ}
Let $\{X_n, n\ge 1\}$ be a sequence of independent and identically distributed random variables with common nonuniform distribution $\{p_i, i\ge 1\}$ satisfying $\ee|\ln p(X_1)|^{2+\delta} < \infty$. Assume that there exists an integer valued function $K(n)$ such that, for some $0< \delta<\frac{1}{2}$,
\beq\label{thmJ-1}
K(n)\le n^{\frac{1}{2}-\delta},\ \ n^{\frac{1}{2}+\delta}\ln n\sum_{i=K(n)}^{\infty}p_i \leq C
\ \ \text{and}\ \ -n^{\frac{1}{2}+\delta}\sum_{i=K(n)}^{\infty}p_i\ln p_i\leq C.
\deq
If there exists an $\epsilon \in (1/2, 1)$ with
\beq\label{thmJ-2}
\sum_{i=1}^{\infty}p_{i}^{1-\epsilon}<\infty,
\deq
then we have
\beq\label{thmJ-3}
\sup_x\left|\pp\left(\frac{\sqrt{n}}{\sigma}(\hat H_{JK}-H)\le x\right)-\Phi(x)\right|\le
Cn^{-\frac{\delta}{2}}+Cn^{\frac{1}{4}-\frac{\epsilon}{2}},
\deq
where $\Phi$ denotes the standard normal distribution function and $\sigma^2=Var(\ln p(X_1))$.
\end{thm}
Next, we provide some examples to verify the conditions of Theorems \ref{thmH}, \ref{thmM}, and \ref{thmJ}.
\begin{exa}
For any $0< \delta<\frac{1}{2}$ and $\lambda>\frac{2}{1-2\delta}$, let $p_i=C_\lambda i^{-\lambda}$ for all $i=1,2, \cdots$, where
$$
C_\lambda=\frac{1}{\sum_{i=1}^\infty i^{-\lambda}}<\infty.
$$
It is easy to check that $\sum_{i=1}^\infty p_i|\ln p_i|^{2+\delta}<\infty$. And let $K(n)=n^{\frac{1}{\lambda}}$, we have
$$
\aligned
&n^{\frac{1}{2}+\delta}\sum_{i=K(n)}^{\infty}p_i\ln n\sim
n^{\frac{1}{2}+\delta}\ln n\int_{K(n)}^{\infty}\frac{C_\lambda}{x^\lambda} dx
=\frac{C_\lambda\ln n}{\lambda-1}n^{\frac{1}{\lambda}-(\frac{1}{2}-\delta)}\leq \frac{C_\lambda}{\lambda-1},
\endaligned
$$
and
$$
\aligned
&-n^{\frac{1}{2}+\delta}\sum_{i=K(n)}^{\infty}p_i\ln p_i\sim
n^{\frac{1}{2}+\delta}\int_{K(n)}^{\infty}\frac{C_\lambda}{x^\lambda}\ln \left(\frac{x^\lambda}{C_\lambda}\right)dx\\
=&n^{\frac{1}{2}+\delta}(K(n))^{1-\lambda}\left(\frac{C_\lambda\lambda\ln K(n)}{\lambda-1}+\frac{ C_\lambda\lambda}{(\lambda-1)^2}-\frac{C_\lambda\ln C_\lambda}{\lambda-1}\right)\\
\sim&\frac{C_\lambda\lambda\ln K(n)}{\lambda-1}n^{\frac{1}{2}+\delta}(K(n))^{1-\lambda}=\frac{C_\lambda\ln n}{\lambda-1}n^{\frac{1}{\lambda}-(\frac{1}{2}-\delta)}\leq \frac{C_\lambda}{\lambda-1}.
\endaligned
$$
Then the sufficient conditions of Theorem \ref{thmH} and Theorem \ref{thmM} hold.

Furthermore, for the $\frac{1}{2}< \epsilon< 1$ and $\lambda>\max\{\frac{2}{1-2\delta},\frac{1}{1-\epsilon}\}$, we have
$$
\sum_{i=1}^\infty p_i^{1-\epsilon}=C_\lambda^{1-\epsilon}\sum_{i=1}^\infty \frac{1}{i^{\lambda(1-\epsilon)}}<\infty.
$$
Then the sufficient conditions of Theorem \ref{thmJ} hold.
\end{exa}
\begin{exa}
For every $i=1,2, \cdots$ and any $\lambda>0$, let $p_i=C_\lambda e^{-\lambda i}$ where
$$
C_\lambda=
\frac{1}{\sum_{i=1}^\infty e^{-\lambda i}}<\infty.
$$
It's easy to check that $\sum_{i=1}^\infty p_i|\ln p_i|^{2+\delta}<\infty$. And for any $0< \delta< \frac{1}{2}$, let $K(n)= \lambda^{-1}\ln n$, we have
$$
\aligned
&n^{\frac{1}{2}+\delta}\sum_{i=K(n)}^{\infty}p_i\ln n\sim
n^{\frac{1}{2}+\delta}\ln n\int_{K(n)}^{\infty}C_\lambda e^{-\lambda x} dx
=\frac{C_\lambda}{\lambda}n^{\delta-\frac{1}{2}}\ln n\leq \frac{C_\lambda}{\lambda},
\endaligned
$$
and
$$
\aligned
&n^{\frac{1}{2}+\delta}\sum_{i=K(n)}^{\infty}\left(-p_i\ln p_i\right)\sim -n^{\frac{1}{2}+\delta}\int_{K(n)}^{\infty}C_\lambda e^{-\lambda x}\ln (C_\lambda e^{-\lambda x})dx\\
=&n^{\frac{1}{2}+\delta}C_\lambda e^{-\lambda K(n)}\left(K(n)+\frac{1}{\lambda}-\frac{\ln C_\lambda}{\lambda}\right)\\
\sim&C_\lambda K(n)n^{\frac{1}{2}+\delta}e^{-\lambda K(n)}
=\frac{C_\lambda}{\lambda}n^{\delta-\frac{1}{2}}\ln n\leq \frac{C_\lambda}{\lambda}.
\endaligned
$$
Then the sufficient conditions of Theorem \ref{thmH} and Theorem \ref{thmM} hold. Furthermore,
$$
\sum_{i=1}^\infty p_i^{1-\epsilon}=C_\lambda^{1-\epsilon}\sum_{i=1}^\infty \frac{1}{e^{(1-\epsilon)\lambda i}}<\infty.
$$
Then the sufficient conditions of Theorem \ref{thmJ} hold.
\end{exa}
\begin{exa}\label{gse}
For every $i=1,2, \cdots$, let $p_i=\frac{C}{i^4(\ln i)^2 }$ and  $K(n)= n^{\frac{1}{2}-\delta}$,
it's easy to check that $\sum_{i=1}^\infty p_i|\ln p_i|^{2+\delta}<\infty$. And for any $0< \delta< \frac{1}{2}$, we have
$$
\aligned
&n^{\frac{1}{2}+\delta}\sum_{i=K(n)}^{\infty}p_i\ln n\sim
n^{\frac{1}{2}+\delta}\ln n\int_{K(n)}^{\infty}\frac{C}{x^4(\ln x)^2} dx
\le C n^{\frac{1}{2}+\delta}\frac{\ln n}{3K^3(n)\ln^2 K(n)}\leq \frac{C}{3\left(\frac{1}{2}-\delta\right)^2},
\endaligned
$$
and
$$
\aligned
&n^{\frac{1}{2}+\delta}\sum_{i=K(n)}^{\infty}\left(-p_i\ln p_i\right)\sim
n^{\frac{1}{2}+\delta}\int_{K(n)}^{\infty}\frac{C}{x^4(\ln x)^2 }\ln \left(\frac{x^4(\ln x)^2 }{C}\right)dx\\
\sim& 4 C n^{\frac{1}{2}+\delta}\int_{K(n)}^{\infty}\frac{1}{x^4\ln x}dx
\le 4 C n^{\frac{1}{2}+\delta}\frac{1}{3K^3(n)\ln K(n)}
\leq \frac{4 C}{3\left(\frac{1}{2}-\delta\right)}.
\endaligned
$$
Thus, the sufficient conditions of Theorems \ref{thmH} and \ref{thmM} are satisfied. Furthermore, for $\frac{1}{2}< \epsilon\le \frac{3}{4}$, we have
$$
\sum_{i=1}^\infty p_i^{1-\epsilon}=C^{1-\epsilon}\sum_{i=1}^\infty \frac{1}{i^{4(1-\epsilon)}(\ln i)^{2(1-\epsilon)}}<\infty.
$$
Then the sufficient conditions of Theorem \ref{thmJ} hold.
\end{exa}
\section{Proofs of main results}
We state some useful lemmas to prove these main results.
\begin{lem}\label{lem1}\cite{C-S}
 Let $X_1, X_2, \cdots, X_n$ be independent and not necessarily identically distributed random variables with zero means and finite variances. Define $W=\sum_{k=1}^nX_k$ and assume that $Var(W)=1$. Let $F$ be the distribution function of $W$ and $\Phi$ be the standard normal distribution function. Then there exists an absolute constant $C$ such that for every real number $x$,
 $$
 |F(x)-\Phi(x)|\le C\sum_{i=1}^n\left(\frac{\ee X_i^2I_{\left\{|X_i|>1+|x|\right\}}}{(1+|x|)^2}+
 \frac{\ee |X_i|^3I_{\{|X_i|\le 1+|x|\}}}{(1+|x|)^3}\right).
 $$
 Furthermore, we have
 $$
 \sup_{x}|F(x)-\Phi(x)|\le C\sum_{i=1}^n\left(\ee X_i^2I_{\{|X_i|>1\}}+
 \ee |X_i|^3I_{\{|X_i|\le 1\}}\right).
 $$
\end{lem}
\begin{lem}\label{lem2} \cite{C-R}
For any random variables $X$, $Y$, real $x$ and constant $a>0$,
$$
\sup_{x}\big|\pp(X+Y\leq x)-\Phi(x)\big|\leq\sup_{x}\big|\pp(X\leq x)-\Phi(x)\big|+\frac{a}{\sqrt{2\pi}}+\pp\left(|Y|>a\right),
$$
where $\Phi(x)$ is the standard normal distribution.
\end{lem}
\begin{lem}\label{lem3} \cite{G-Z}
Assume that the function $g: [0,1] \rightarrow \rr$ is differentiable on $[0,1]$ and its derivative $g'$ is $\beta$-H\"{o}lder
continuous, then for any $a \in (0,1]$ we can write
$$ g(x)=g(a)+g'(a)(x-a)+R_a(x),$$
where
$$|R_a(x)|\leq C|x-a|^{\beta+1}$$
for some $C > 0$.
\end{lem}
\begin{lem}\label{lem4} \cite{G-Z}
If $\beta \in (0,1]$, then
\beq\label{lem4-1}
\ee[|\hat p_{n,i}-p_{n,i}|^{\beta+1}]\leq 4p_{n,i} n^{-\beta}
\deq
and there is a constant $C>0$ depending only on $\beta$ with
\beq\label{lem4-2}
\ee[|\hat p_{n,i}-p_{n,i}|^{\beta+1}]\leq C p_{n,i}^{(\beta+1)/2} n^{-(\beta+1)/2}.
\deq
\end{lem}
\begin{lem}\label{lemJ}\cite{C-M}
Let $\{X_n, n\ge 1\}$ be a sequence of independent and identically distributed Bernoulli random variables with parameter $p\in (0,1)$. For $m=1,\dots,n$ let $S_m=\sum_{i=1}^m X_i$, $\hat p_m=S_m/m$, $\hat h_m=-\hat p_m \ln{\hat p_m}$. Then for $n \geq 3$, we have
\beq\label{lemJ-1}
\ee\left(\hat h_n-\hat h_{n-1}\right)\leq\frac{2p}{(n-1)\left((n-2)p+2\right)}.
\deq
\end{lem}
For fixed $\varepsilon>0$, rewriting the upper bound of (\ref{lemJ-1}) gives
\beq\label{lemJ-2}
\frac{2p}{(n-1)\left((n-2)p+2\right)}=\frac{2p^{1-\varepsilon}}{n-1}
\left(\frac{p^\varepsilon}{(n-2)p+2}\right)=:
\frac{2p^{1-\varepsilon}}{n-1}g(n,p,\varepsilon).
\deq
\begin{lem}\label{lemJ2}\cite{C-M}
For any $\varepsilon \in (0,1)$ and $n \geq 3$, there exists a $p_0 \in (0,1)$ such that $g(n,p,\varepsilon)$ defined in (\ref{lemJ-2}) is maximized at $p_0$ and
\beq\label{lemJ-3}
0\leq g(n,p_0,\varepsilon)= \mathcal{O}(n^{-\varepsilon}).
\deq
\end{lem}
\begin{lem}\label{lemJ3}\cite{C-M}
For any distribution, let $B_n=\ee \hat H_n-H$ be the bias of the plug-in based on a sample of size $n$, we have
$$
\hat B_{JK}:=\frac{n-1}{n}\sum_{j=1}^n \left(\hat H_n-\hat H_n^{(j)}\right)\ge 0
$$
and
\beq\label{lemJ-4}
\ee(\hat B_{JK})=(n-1)(B_n-B_{n-1}).
\deq
\end{lem}
\begin{proof} [{\bf Proof of Theorem \ref{thmtheta}}]
From Lemma \ref{lem3}, we have
\beq\label{thm1-p1}
\aligned
\hat \theta_n-\theta_n=\sum_{i=1}^\infty g'(p_{n,i})(\hat p_{n,i}-p_{n,i})+\sum_{i=1}^\infty R_{p_{n,i}}(\hat p_{n,i}).
\endaligned
\deq
For every $n\ge 1$ and $1\le k\le n$, let us define
\beq\label{thm1-p2}
T_{k,n}:=\sum_{i=1}^{\infty}\left(I_{\{X_{k, n}=a_i\}}-p_{n,i}\right)g'(p_{n,i}),
\deq
then we have
\beq\label{thm1-p3}
\aligned
\sum_{i=1}^\infty g'(p_{n,i})(\hat p_{n,i}-p_{n,i})=\frac{1}{n}\sum_{k=1}^nT_{k,n}.
\endaligned
\deq
It is easy to check that $\{T_{k,n}\}$ is an array of independent and identically distributed random variables with $\ee T_{1,n}=0$ and
\begin{align*}
Var(T_{1,n})=&\ee\left(\sum_{i=1}^{\infty}\left(1_{\{X_{1, n}=a_i\}}-p_{n,i}\right)g'(p_{n,i})\right)^2\\
=&\sum_{i=1}^{\infty}p_{n,i}(g'(p_{n,i})^2-\left(\sum_{i=1}^{\infty}p_{n,i}g'(p_{n,i})\right)^2
=\sigma^2_n.
\end{align*}
Since $g'$ is $\beta$-H\"{o}lder continuous, there exists a positive constant $M$, such that
\beq\label{g}
|T_{k,n}|\le M\sum_{i=1}^{\infty}(I_{\{X_{k, n}=a_i\}}+p_{n,i})\le 2M.
\deq
Thus for any $0< \alpha< 1$, using H\"older's inequality, we have
\begin{align}\label{thm1-p4}
\ee|T_{1,n}|^{2+\alpha}
 \le&
 \left(\ee|T_{1,n}|^{\frac{2}{1-\alpha}}\right)^{\frac{1-\alpha}{2}}
 \left(\ee|T_{1,n}|^2\right)^{\frac{1+\alpha}{2}}
 \le C\sigma_n^{1+\alpha}.
\end{align}
By Lemma \ref{lem2}, for any $a>0$, we have
\begin{align}\label{thm1-p5}
&\sup\limits_x\left|\pp\left(\frac{\sqrt{n}}{\sigma_n}(\hat \theta_n-\theta_n)<x\right)-\Phi(x)\right|\nonumber\\
\le& \sup\limits_x\left|\pp\left(\frac{1}{\sqrt{n}\sigma_n}\sum_{k=1}^nT_{k,n}<x\right)-\Phi(x)\right|
+\frac{a}{\sqrt{2\pi}}+\pp\left(\frac{\sqrt{n}}{\sigma_n}\left|\sum_{i=1}^\infty R_{p_{n,i}}(\hat p_{n,i})\right|>a\right).
\end{align}
Firstly, by Lemma \ref{lem1} and \ref{thm1-p4}, for any $0< \alpha=2\gamma< 1$ we have
\begin{align}\label{thm1-p6}
&\sup\limits_x\left|\pp\left(\frac{1}{\sqrt{n}\sigma_n}\sum_{k=1}^nT_{k,n}<x\right)-\Phi(x)\right|\nonumber\\
\le &C\sum_{k=1}^n\left\{\ee \left(\frac{T_{k,n}}{\sqrt{n}\sigma_n}\right)^2I_{\left\{|T_{k,n}|>\sqrt{n}\sigma_n\right\}}+
 \ee\left(\frac{|T_{k,n}|}{\sqrt{n}\sigma_n}\right)^3I_{\left\{|T_{k,n}|\le\sqrt{n}\sigma_n\right\}}\right\}\nonumber\\
 \le & C\sum_{k=1}^n\left\{\ee \left|\frac{T_{k,n}}{\sqrt{n}\sigma_n}\right|^{2+\alpha}I_{\left\{|T_{k,n}|>\sqrt{n}\sigma_n\right\}}+
 \ee\left|\frac{T_{k,n}}{\sqrt{n}\sigma_n}\right|^{2+\alpha}I_{\left\{|T_{k,n}|\le\sqrt{n}\sigma_n\right\}}\right\}\nonumber\\
 \le & \frac{C}{n^{\alpha/2}\sigma_n^{2+\alpha}}\ee|T_{1,n}|^{2+\alpha}
 \le \frac{C}{n^{\alpha/2}\sigma_n}=\frac{C}{n^{\gamma}\sigma_n}.
\end{align}
By using Markov inequality and Lemma \ref{lem3}, we have
\begin{align}\label{thm1-p7}
&\pp\left(\frac{\sqrt{n}}{\sigma_n}\left|\sum_{i=1}^\infty R_{p_{n,i}}(\hat p_{n,i})\right|>a\right)
\leq \frac{C\sqrt{n}}{a\sigma_n}\sum_{i=1}^\infty\ee\left(|\hat p_{n,i}-p_{n,i}|^{\beta+1}\right).
\end{align}
Next we can get the further upper bound of the above inequation by Lemma \ref{lem4}.
When $\beta \in (0.5,1]$, using (\ref{lem4-1}), we can obtain the upper bound as
\beq\label{thm1-p8}
\frac{C\sqrt{n}}{a\sigma_n}\sum_{i=1}^\infty 4n^{-\beta}p_{n,i}=\frac{C}{a \sigma_n n^{(\beta-1/2)}},
\deq
and when $\beta \in (0,0.5]$, using (\ref{lem4-2}) and the condition (\ref{thm1-c}), we obtain the upper bound as
\beq\label{thm1-p9}
\frac{C\sqrt{n}}{a\sigma_n}\sum_{i=1}^\infty p_{n,i}^{(\beta+1)/2} n^{-(\beta+1)/2}=\frac{C}{a \sigma_n n^{(\beta/2)}}.
\deq
By taking $a=\frac{1}{n^{\gamma/2}\sigma_n^{1/2}}$,  (\ref{thmtheta-3}) holds by using (\ref{thm1-p5}), (\ref{thm1-p6}), (\ref{thm1-p7}), (\ref{thm1-p8}) and (\ref{thm1-p9}).
\end{proof}
\begin{proof}[{\bf Proof of Corollary \ref{cortheta}}]
Corollary $\ref{cortheta}$ corresponds to the case where $\sigma_n=\sigma$ is fixed in Theorem $\ref{thmtheta}$. Since $\sigma^2 > 0$ for any nonuniform distribution, the proof follows directly from that of Theorem $\ref{thmtheta}$ with this parameter substitution.
\end{proof}
\begin{proof}[{\bf Proof of Theorem \ref{thmH}}]
Let $\{X_{k,n}, 1\le k\le n, n\ge 1\}$ be an array of independent and identically distributed random variables with common nonuniform distribution $\{p_{n,i},  1\le i\le K(n), n\ge 1\}$, i.e., for any $n\ge 1$,
$$
\pp(X_{1,n}=i)=p_{n,i},  \ \ \  \ 1\le i\le K(n)
$$
where
$$
p_{n,i}=
\begin{cases}
p_i,& \ \ \ 1\leq i<K(n)\\
\sum_{i\geq K(n)}p_i,& \ \ \ i=K(n).
\end{cases}
$$
Define the entropy as
$$
H_n=-\sum_{i=1}^{K(n)} p_{n,i}\ln p_{n,i},
$$
and the corresponding sample proportion is
$$
\hat p_{n,i}=\frac{1}{n}\sum_{k=1}^nI_{\{X_{k,n}=a_i\}}.
$$
From the definition of the plug-in estimator $\hat H_n$ for the entropy $H$, we have
\beq\label{thmH-p1}
\aligned
\hat H- H=&-\sum_{i=1}^{\infty}(\hat p_i-p_i)\ln p_i -\sum_{i=1}^{\infty}\hat p_i\ln \frac{\hat p_i}{p_i}.
\endaligned
\deq
For every $n\ge 1$ and $1\le k\le n$, let us define
$$
T_{k}:=-\sum_{i=1}^{\infty}\left(I_{\{X_{k}=a_i\}}-p_i\right)\ln p_i,
$$
then we have
\beq\label{thmH-p2}
\aligned
-\sum_{i=1}^{\infty}(\hat p_i-p_i)\ln p_i =\frac{1}{n}\sum_{k=1}^nT_{k}.
\endaligned
\deq
It's easy to see that $\ee|\ln p(X_1)|^{2+\delta}=\sum_{i=1}^{\infty} p_i\left|\ln p_i\right|^{2+\delta} < \infty$ can imply
$H=-\sum_{i=1}^\infty p_i\ln p_i< \infty$ and $\sigma^2 < \infty$.
Using the elementary inequality $(a+b)^{2+\delta}\le C(a^{2+\delta}+b^{2+\delta})$, we have
\beq\label{thmH-p3}
\aligned
\ee|T_{1}|^{2+\delta}=&\ee\left|\sum_{i=1}^{\infty}\left(I_{\{X_1=a_i\}}-p_i\right)\ln p_i\right|^{2+\delta}\\
\le &C\ee\left|\sum_{i=1}^{\infty}I_{\{X_{1}=a_i\}}\ln p_i\right|^{2+\delta}+C\left|-\sum_{i=1}^{\infty}p(i)\ln p_i\right|^{2+\delta}\\
\leq &C\sum_{i=1}^{\infty} p_i\left|\ln p_i\right|^{2+\delta}+ C\left|H\right|^{2+\delta}< \infty.
\endaligned
\deq
From (\ref{thmH-p1}), (\ref{thmH-p2}) and Lemma \ref{lem2}, for any $a>0$, we have
\begin{align}\label{thmH-p4}
&\sup\limits_x\left|\pp\left(\frac{\sqrt{n}}{\sigma}(\hat{H}-H)<x\right)-\Phi(x)\right|\nonumber\\
\le &\sup\limits_x\left|\pp\left(\frac{1}{\sqrt{n}\sigma}\sum_{k=1}^nT_{k}<x\right)-\Phi(x)\right|
+\frac{a}{\sqrt{2\pi}} +\pp\left(\frac{\sqrt{n}}{\sigma}\left|\sum_{i=1}^{\infty}\hat p_{i}\ln {\frac{\hat p_{i}}{p_i}}\right|>a\right).
\end{align}
Firstly, by (\ref{thmH-p3}) and Lemma \ref{lem1}, for the fixed $0\le \delta< \frac{1}{2}$, we have
\begin{align}\label{thmH-p5}
&\sup\limits_x\left|\pp\left(\frac{1}{\sqrt{n}\sigma}\sum_{k=1}^nT_k<x\right)-\Phi(x)\right|\nonumber\\
\le &C\sum_{k=1}^n\left\{\ee \left(\frac{T_k}{\sqrt{n}\sigma}\right)^2I_{\left\{|T_k|>\sqrt{n}\sigma\right\}}+
 \ee\left(\frac{T_k}{\sqrt{n}\sigma}\right)^3I_{\left\{|T_k|\le\sqrt{n}\sigma\right\}}\right\}\nonumber\\
 \le & C\sum_{k=1}^n\left\{\ee \left|\frac{T_k}{\sqrt{n}\sigma}\right|^{2+\delta}I_{\left\{|T_k|>\sqrt{n}\sigma\right\}}+
 \ee\left|\frac{T_k}{\sqrt{n}\sigma}\right|^{2+\delta}I_{\left\{|T_k|\le\sqrt{n}\sigma\right\}}\right\}\nonumber\\
 \le & \frac{C}{n^{\delta/2}\sigma^{2+\delta}}\ee|T_1|^{2+\delta}
  \le  \frac{C}{n^{\delta/2}}.
\end{align}
We observe that
$-\hat p_{i}\ln \hat p_{i}\leq\hat p_{i}\ln n$ and $-\hat p_{n,i}\ln \hat p_{n,i}\leq\hat p_{n,i}\ln n$, then the following two inequalities
$$
0\leq\sqrt{n}\sum_{i\geq K(n)}\left(-\hat p_{i}\ln {\hat p_{i}}\right)\leq\sqrt{n}\sum_{i\geq K(n)}\hat p_{i}\ln n
$$
and
\begin{align}\label{thmH-p6}
0\leq-\sqrt{n}\hat p_{n,K(n)}\ln \hat p_{n,K(n)}\leq\sqrt{n}\hat p_{n,K(n)}\ln n
\end{align}
clearly hold.
From the following inequality
$$
\frac{h}{1+h}<\ln (1+h)<h \ \ \ \text{for}\ \ h>-1,
$$
we have
\beq
\sum_{i=1}^{K(n)}\hat p_{n,i}\ln \frac{\hat p_{n,i}}{p_{n,i}}\le \sum_{i=1}^{K(n)}\hat p_{n,i}\left(\frac{\hat p_{n,i}}{p_{n,i}}-1\right)=\sum_{i=1}^{K(n)}\frac{\left(\hat p_{n,i}-p_{n,i}\right)^2}{p_{n,i}}
\deq
and
\beq
\sum_{i=1}^{K(n)}\hat p_{n,i}\ln \frac{\hat p_{n,i}}{p_{n,i}}\ge \sum_{i=1}^{K(n)}\frac{\hat p_{n,i}\left(\frac{\hat p_{n,i}}{p_{n,i}}-1\right)}{1+\left(\frac{\hat p_{n,i}}{p_{n,i}}-1\right)}
=\sum_{i=1}^{K(n)} p_{n,i}\left(\frac{\hat p_{n,i}}{p_{n,i}}-1\right)=0.
\deq
With the condition $K(n)\le n^{\frac{1}{2}-\delta}$, we obtain
\begin{align*}
&\pp\left(\frac{\sqrt{n}}{\sigma}\left|\sum_{i=1}^{K(n)}\hat p_{n,i}\ln {\frac{\hat p_{n,i}}{p_{n,i}}}\right|>a\right)\\
\le&\pp\left(\frac{\sqrt{n}}{\sigma}\left|\sum_{i=1}^{K(n)}\frac{\left(\hat p_{n,i}-p_{n,i}\right)^2}{p_{n,i}}\right|>a\right)\\
\le&\frac{\sqrt{n}}{a \sigma}\sum_{i=1}^{K(n)}\ee \left(\frac{\left(\hat p_{n,i}-p_{n,i}\right)^2}{p_{n,i}}\right)\\
=&\frac{1}{n^{3/2}a\sigma}\sum_{i=1}^{K(n)} \frac{1}{p_{n,i}}\ee\left[\left(\sum_{k=1}^n(I_{\{X_{k,n}=a_i\}} -p_{n,i})\right)^2\right]\\
=& \frac{1}{a\sqrt{n}\sigma}\sum_{i=1}^{K(n)}(1-p_{n,i}) \le\frac{K(n)}{a\sqrt{n}\sigma}\le\frac{C}{an^{\delta}}.
\end{align*}
From the definition that $p_{n,K(n)}=\sum_{i\geq K(n)}p_i$, we have
$$
\sum_{i\geq K(n)} p_i\ln n=p_{n,K(n)}\ln n
$$
and
$$
\sum_{i\geq K(n)} p_i\ln p_i\le \sum_{i\geq K(n)} p_i\ln p_{n,K(n)}= p_{n,K(n)}\ln p_{n,K(n)}.
$$
With (\ref{thmH-p6}) and Markov's inequality, we have
\begin{align*}
&\pp\left(\frac{\sqrt{n}}{\sigma}\left|\hat p_{n,K(n)}\ln \frac{\hat p_{n,K(n)}}{p_{n,K(n)}}\right|>a\right)\\
\le&\pp\left(-\hat p_{n,K(n)}\ln \hat p_{n,K(n)}-\hat p_{n,K(n)}\ln{p_{n,K(n)}}>\frac{a\sigma}{\sqrt{n}}\right)\\
\le&\pp\left(\hat p_{n,K(n)}\ln n-\hat p_{n,K(n)}\ln{p_{n,K(n)}}>\frac{a\sigma}{\sqrt{n}}\right)\\
\le&\frac{\sqrt{n}}{a\sigma}\ee\left|\hat p_{n,K(n)}\ln n-\hat p_{n,K(n)}\ln{p_{n,K(n)}}\right|\\
=&\frac{\sqrt{n}}{a\sigma}\left(p_{n,K(n)}\ln n-p_{n,K(n)}\ln p_{n,K(n)}\right)\\
\le&\frac{\sqrt{n}}{a\sigma}\left(p_{n,K(n)}\ln n-\sum_{i\geq K(n)} p_i\ln p_i\right)
\end{align*}
and
\begin{align*}
&\pp\left(\frac{\sqrt{n}}{\sigma}\left|\sum_{i\geq K(n)}\hat p_{i}\ln {\frac{\hat p_{i}}{p_i}}\right|>a\right)\\
\le&\pp\left(-\sum_{i\geq K(n)}\hat p_{i}\ln \hat p_{i}-\sum_{i\geq K(n)}\hat p_{i}\ln p_{i}>\frac{a\sigma}{\sqrt{n}}\right)\\
\le&\pp\left(\sum_{i\geq K(n)}\hat p_{i}\ln n-\sum_{i\geq K(n)}\hat p_{i}\ln p_{i}>\frac{a\sigma}{\sqrt{n}}\right)\\
\le&\frac{\sqrt{n}}{a\sigma}\ee\left|\sum_{i\geq K(n)}\hat p_{i}\ln n-\sum_{i\geq K(n)}\hat p_{i}\ln p_{i}\right|\\
\le&\frac{\sqrt{n}}{a\sigma}\left(p_{n,K(n)}\ln n-\sum_{i\geq K(n)} p_i\ln p_i\right).
\end{align*}
Then using the assumption (\ref{thmH-0}), we get
\begin{align}\label{thmH-p7}
&\pp\left(\frac{\sqrt{n}}{\sigma}\left|\sum_{i=1}^{\infty}\hat p_{i}\ln {\frac{\hat p_{i}}{p_i}}\right|>a\right)\nonumber\\
=&\pp\left(\frac{\sqrt{n}}{\sigma}\left|\sum_{i=1}^{K(n)}\hat p_{n,i}\ln {\frac{\hat p_{n,i}}{p_{n,i}}}-\hat p_{n,K(n)}\ln \frac{\hat p_{n,K(n)}}{p_{n,K(n)}}+\sum_{i\geq K(n)}\hat p_{i}\ln {\frac{\hat p_{i}}{p_i}}\right|>a\right)\nonumber\\
\le &C\left(\frac{1}{a n^{\delta}}+\frac{\sqrt{n}}{a}p_{n,K(n)}\ln n
-\frac{\sqrt{n}}{a}\sum_{i\geq K(n)} p_i\ln p_i\right)
\le C\frac{1}{a n^{\delta}}.
\end{align}
By taking $a=n^{-\frac{\delta}{2}}$,  (\ref{thmH-1}) holds by using (\ref{thmH-p4}), (\ref{thmH-p5}) and (\ref{thmH-p7}).
\end{proof}
\begin{proof} [{\bf Proof of Theorem \ref{thmM}}]
By the defition of $H_{MM}$, we have
\beq\label{thmM-p1}
\hat H_{MM}-H=\hat H_n-H+\frac{\hat m-1}{2n}.
\deq
By Lemma \ref{lem2}, for any $a>0$ we have
\begin{align}\label{thmM-p2}
&\sup_x\left|\pp\left(\frac{\sqrt{n}}{\sigma}(\hat H_{MM}-H)\le x\right)-\Phi(x)\right|\nonumber\\
\le& \sup\limits_x\left|\pp\left(\frac{\sqrt{n}}{\sigma}(\hat H-H)\le x\right)-\Phi(x)\right|
+\frac{a}{\sqrt{2\pi}} +\pp\left(\frac{\sqrt{n}}{\sigma}\left|\frac{\hat m-1}{2n}\right|>a\right).
\end{align}
Firstly, by Theorem \ref{thmH}
\beq\label{thmM-p3}
\sup\limits_x\left|\pp\left(\frac{\sqrt{n}(\hat{H}-H)}{\sigma}<x\right)-\Phi(x)\right|
 \leq Cn^{-\frac{\delta}{2}}.
\deq
Then from Bernoulli's inequality
$(1-p_i)^n\ge 1-np_i$, we have
$$
1-(1-p_i)^n\le np_i.
$$
By using Markov inequality and the assumption (\ref{thmM-1}), we have
\begin{align}\label{thmM-p4}
\pp\left(\frac{\sqrt{n}}{\sigma}\left|\frac{\hat m-1}{2n}\right|>a\right)
&\leq \frac{C}{a\sqrt{n}}\left(\ee|\hat m|+1\right)
=\frac{C}{a\sqrt{n}}\left(1+\sum_{i=1}^{\infty}\left(1-(1-p_i)^n\right)\right)\nonumber\\
&\leq \frac{C}{a\sqrt{n}}K(n)+\frac{C\sqrt{n}}{a}\sum_{i=K(n)}^{\infty}p_i
\leq \frac{C}{a n^\delta}.
\end{align}
By taking $a=n^{-\frac{\delta}{2}}$,  (\ref{thmM-3}) holds by using (\ref{thmM-p2}), (\ref{thmM-p3}) and (\ref{thmM-p4}).
\end{proof}
\begin{proof} [{\bf Proof of Theorem \ref{thmJ}}]
By the defition of $H_{JK}$, we have
\beq\label{thmJ-p1}
\hat H_{JK}-H=\hat H_n-H+\frac{n-1}{n}\sum_{j=1}^n \left(\hat H-\hat H^{(j)}\right).
\deq
Let us define
$$
\hat B_{JK}:=\frac{n-1}{n}\sum_{j=1}^n \left(\hat H-\hat H^{(j)}\right).
$$
By Lemma \ref{lem2}, for any $a>0$ we have
\begin{align}\label{thmJ-p2}
&\sup_x\left|\pp\left(\frac{\sqrt{n}}{\sigma}(\hat H_{JK}-H)\le x\right)-\Phi(x)\right|\nonumber\\
\le& \sup\limits_x\left|\pp\left(\frac{\sqrt{n}}{\sigma}(\hat H-H)\le x\right)-\Phi(x)\right|
+\frac{a}{\sqrt{2\pi}}+\pp\left(\frac{\sqrt{n}}{\sigma}\left|\hat B_{JK}\right|>a\right).
\end{align}
Firstly, by Theorem \ref{thmH} we have
\beq\label{thmJ-p3}
\sup\limits_x\left|\pp\left(\frac{\sqrt{n}(\hat{H}-H)}{\sigma}<x\right)-\Phi(x)\right|
 \leq Cn^{-\frac{\delta}{2}}.
\deq
For any distribution, let $B_n=\ee\hat H_n-H$ be the bias of the plug-in based on a sample of size $n$. And for every $i$ and every $m\leq n$, let
$$
S_{m,i}=\sum_{k=1}^mI_{\{X_k=i\}}\ \ and\ \ \hat H_m=-\sum_{i=1}^\infty \frac{S_{m,i}}{m}\ln{\frac{S_{m,i}}{m}}
$$
be the observed letter counts and the plug-in estimator of entropy based on the first $m$ observations. By using Lemma \ref{lemJ}, Lemma \ref{lemJ2} and the $\epsilon \in (1/2, 1)$ satisfies assumption (\ref{thmJ-2}), we have
\begin{align*}
B_n-B_{n-1}&=\ee(\hat H_n-\hat H_{n-1})\nonumber\\
&=\sum_{i=1}^\infty \ee\left(\frac{S_{n-1,i}}{n-1}\ln{\frac{S_{n-1,i}}{n-1}}-\frac{S_{n,i}}{n}\ln{\frac{S_{n,i}}{n}}\right)\nonumber\\
&\leq \sum_{i=1}^\infty \frac{2p_i}{(n-1)\left((n-2)p_i+2\right)}\nonumber\\
&=2\sum_{i=1}^\infty p_i^{1-\epsilon}\frac{1}{(n-1)}\frac{p_i^\epsilon}{\left((n-2)p_i+2\right)}\nonumber\\
&\leq \frac{C}{n^{1+\epsilon}}.
\end{align*}
By using Markov inequality and Lemma \ref{lemJ3}, we have
\begin{align}\label{thmJ-p4}
\pp\left(\frac{\sqrt{n}}{\sigma}\left|\hat B_{JK}\right|>a\right)
&\leq \frac{\sqrt{n}}{a\sigma}\ee\left|\hat B_{JK}\right|
=\frac{\sqrt{n}(n-1)}{a\sigma}\left(B_n-B_{n-1}\right)\nonumber\\
&\leq \frac{C}{a}n^{1/2-\epsilon}.
\end{align}
By taking $a=n^{\frac{1}{4}-\frac{\epsilon}{2}}$, (\ref{thmJ-3}) holds by using (\ref{thmJ-p2}), (\ref{thmJ-p3}) and (\ref{thmJ-p4}).
\end{proof}

\end{document}